\documentclass{article}
\usepackage{amsmath, amssymb, latexsym, theorem}
\usepackage{subfigure}
\numberwithin{equation}{section}

\theoremstyle{plain}
\theorembodyfont{\itshape}
\newtheorem{theorem}{Theorem}[section]
\newtheorem{proposition}[theorem]{Proposition}

\newtheorem{corollary}[theorem]{Corollary}
                                                                                
\theorembodyfont{\rmfamily}
\newtheorem{definition}[theorem]{Definition}

\newtheorem{example}[theorem]{Example}
\newtheorem{remark}[theorem]{Remark}

\newenvironment{proof}{{\noindent \textbf{Proof}\,\,}}{\hspace*{\fill}$\Box$\medskip}
\title{Total positivity, Grassmannian and modified Bessel functions}
\author{V.M.Buchstaber\thanks{Permanent address: Steklov Mathematical Institute, 8, Gubkina street, 119991, Moscow, Russia.  Email: buchstab@mi.ras.ru}
\thanks{All-Russian Scientific Research Institute for Physical and Radio-Technical Measurements (VNIIFTRI),}
\thanks{Supported by part by RFBR grant 14-01-00506.}, 
A.A.Glutsyuk\thanks{ CNRS, France (UMR 5669 (UMPA, ENS de Lyon) and Interdisciplinary Scientific Center J.-V.Poncelet). 
Email:
aglutsyu@ens-lyon.fr}, 
\thanks{National Research University Higher School of Economics (HSE), Moscow, Russia}
 \thanks{Supported by part by RFBR grants 16-01-00748, 16-01-00766.}}
\begin{document}
\maketitle
\def\zz{\mathbb Z}
\def\nn{\mathbb N}
\def\var{\varepsilon}
\def\td{\mathbb T^3} 
\def\rr{\mathbb R}
\def\la{\lambda}
\def\go#1{\EuFrak#1}
\def\wt{\widetilde}
\def\sign{\operatorname{sign}}
\def\cc{\mathbb C}
\def\diag{\operatorname{diag}}
\def\dd{\Delta_{discr}}

\centerline{\it To Selim Grigorievich Krein with gratitude for his}
\centerline{\it mathematical results and Voronezh Winter Mathematical Schools}

\begin{abstract} A rectangular matrix is called {\it  totally positive,}  (according to F.R.Gantmacher and M.G.Krein) if all its minors are positive. 
A point of a real Grassmanian manifold $G_{l,m}$ of $l$-dimensional subspaces in $\rr^m$ is called {\it strictly totally positive} 
(according to A.E.Postnikov) if one can normalize its Pl\"ucker coordinates to make all of them positive. 
Clearly if a $k\times m$-matrix, $k<m$, is totally positive, then each collection of its $l\leq k$ rows generates an $l$-subspace 
represented by a strictly totally positive point of the Grassmanian manifold $G_{l,m}$. 
The totally positive matrices and the strictly 
totally positive Grassmanians, that is, the subsets of strictly totally positive points in  Grassmanian manifolds 
arise in many domains of mathematics, mechanics and physics. F.R.Gantmacher and M.G.Krein 
considered totally positive matrices in the context of classical mechanics.  S.Karlin considered them in a wide context of analysis, 
differential equations and probability theory. In his well-known book "Total positivity" he provides a construction of totally positive 
matrices coming from Green functions of linear differential operators, including famous class of Sturm--Liouville operators. Total positivity was used for construction of solutions of the  Kadomtsev-Petviashvili (KP) 
partial differential equation in a paper  by T.M.Malanyuk, 
 in a joint paper by M.Boiti, F.Pemperini and A.Pogrebkov and in a joint paper of Y.Kodama and L.Williams.  
 Different problems of mathematics, mechanics and physics led to constructions of totally positive matrices 
 due to many mathematicians, including F.R. Gantmacher, M.G.Krein, I.J.Schoenberg, S.Karlin, 
A.E.Postnikov and ourselves.  In our case totally positive matrices were constructed for solution of problems on model of the overdamped Josephson effect in superconductivity and double confluent Heun equations. 
In our previous paper we have proved that certain determinants formed by modified Bessel functions of the first 
kind are positive on the positive semi-axis. This yields a one-dimensional family of strictly totally positive points  in all the Grassmanian manifolds. 

In the present paper we provide a construction of multidimensional families of strictly totally positive points in all the Grassmanian manifolds, 
again using modified Bessel functions of the first kind but different from the above-mentioned construction. 
These families represent images of explicit injective mappings of the convex open 
subset $\{ x=(x_1,\dots,x_l)\in\rr^l \ | \ 0<x_1<\dots<x_l\}\subset\rr^l$ to the Grassmanian manifolds $G_{l,m}$, $l<m$. 

We present a new result  that gives totally positive matrices formed by values of 
modified Bessel functions with  non-negative integer indices. Their columns are numerated by the indices of the modified Bessel functions, 
and their rows are numerated by their arguments.

S.Karlin presented in his book constructions of totally non-negative matrices given in terms of either 
just one modified Bessel function $I_{\alpha}$, or Green functions of linear differential operators. 
We briefly describe Karlin's results and demonstrate that our results are not covered by his constructions.
\end{abstract}

\section{Introduction}
\subsection{Brief survey on totally positive matrices. Main result} 

The following notion was introduced in the classical books \cite{gantkr, gantkr2} in the context of the classical mechanics. 
\begin{definition} \cite{abgryn, gantkr, pincus}, \cite[p.289 of the russian edition]{gantkr2} A rectangular $l\times m$-matrix  is called {\it  totally positive} (nonnegative), 
if all its minors of all the dimensions  are positive (nonnegative). 
\end{definition}

\begin{example} It is known that every generalized Vandermonde matrix 
$$(f(x_i,y_j))_{i=1,\dots,m; \ j=1,\dots,n}, \ f(x,y)=x^y,$$
$$ 0<x_1<\dots<x_m, \ 0\leq y_1<y_2<\dots<y_n$$
is totally positive, see \cite[chapter XIII, section 8]{gant}.
\end{example}

The study of $n\times n$ matrices with positive elements goes back to Perron \cite{p} who had shown that for such a matrix 
the eigenvalue that is largest in the modulus is simple, real and positive, and the corresponding eigenvector can be normalized to have 
all the components positive (1907). Later in 1908 his result was generalized by G.Frobenius \cite[chapter 13, section 2]{gant} to  those 
matrices with non-negative coefficients that are block-non-decomposable.  
For each one of these matrices he had proved that  its complex eigenvalues of maximal modulus 
are roots of a polynomial $P(\la)=\la^h-r^h$, all of them are simple and  one of them is real and positive. 

In 1935--1937 F.R.Gantmacher and M.G.Krein \cite{gantkr-1, gantkr-0} 
observed that  if the matrix under question satisfies a stronger condition, that is, total positivity (in fact a weaker, 
oscillation property is sufficient), 
then {\it all} its eigenvalues are simple, real and positive. 

Earlier in 1930 I.Schoenberg \cite{schon} studied a more general class of 
 matrices  including totally positive ones: namely, the $m\times n$-matrices such that for 
 every $k\leq\min\{ m , n\}$ all the non-zero minors of  order $k$ have the same sign (either all positive, or all negative). 
 He proved important results relating the latter property with variation-diminishing property of the corresponding linear 
 transformations $\rr^n\to\rr^m$. (Recall that a linear  transformation  $\rr^n\to\rr^m$ is called 
 {\it variation-diminishing,} if it does not increase the number of sign changes in the sequence of coordinates of a vector.) 
   Further results in this direction were obtained by Motzkin in his dissertation in 1933 
 and most complete results were obtained by Gantmacher and Krein \cite{gantkr2}. 
 
 A  two-sided sequence $(a_j)_{j\in\zz}$ of real numbers is called {\it totally nonnegative (positive),} if the 
 infinite matrix $(a_{j-i})_{i,j\in\zz}$ is totally nonnegative (positive). There is a remarkable result on 
 characterization of totally nonnegative sequences. It says that for each totally nonnegative sequence 
 distinct from a geometric progression the corresponding generating function 
 $\sum_j a_jz^j$ converges as a Laurent series in some annulus, extends meromorphically to all of 
 $\cc^*$ and is a product of $\exp(q_1z+q_2z^{-1})$ times an infinite product of fractions of appropriate 
 linear functions; here  $q_1,q_2\geq0$.  In 1948 I.Schoenberg \cite{schon2} proved sufficience: total nonnegativity  of 
 Laurent series of the latter  functions. 
  The converse (necessity) was proved by A.Edrei in 1953 \cite{edrei}. 
 See also \cite[theorem 8.9.5]{karlin} and references in this books and in papers \cite{schon2, edrei}. 
Our result obtained in \cite{bg} is a  result  towards the characterization of {\it totally positive} two-sided sequences. 
In \cite{bg} we have proved total positivity of the sequence 
 of values $(I_j(x))_{j\in\zz}$ of the modified Bessel functions $I_j(z)$ of the first kind for every $x>0$. A weaker statement,  total 
 non-negativity follows from Schoenberg's  theorem \cite{schon2}. 
  
 The characterization of all the totally nonnegative two-sides sequences is a basic fundamental result used 
 in  the description of the characters of representations of the infinite unitary group $U(\infty)=\underrightarrow\lim U(n)$ 
 and the infinite symmetric group $S(\infty)=\underrightarrow\lim S(n)$. See papers by E.Toma \cite{thoma}, 
 D.Voiculescu \cite{voic},  joint papers by A.M.Vershik and S.V.Kerov \cite{vershik1, vershik2} and references therein. 
 
Many results on characterization and properties of  totally 
positive matrices and their relations to other domains of mathematics (e.g., combinatorics, dynamical systems, 
geometry and topology, probability theory, Fourier analysis, representation theory), mechanics and physics 
are given in \cite{abgryn, edrei, gant, gantkr-1, gantkr-0, gantkr, gantkr2, pincus, faljon, faljon2, gm96, karlin, lustig1, lustig2, lustig3, schon2, vershik1, vershik2} and in 
\cite{post, berfomzel, berzel, fomzel, gelf, fomin, pog, kw} (see also references in all these papers and books). 
F.R.Gantmacher and M.G.Krein 
\cite{gantkr, gantkr2} considered totally positive matrices in the context of applications to mechanical problems. S.Karlin \cite{karlin} 
considered them in a wide context of analysis, differential equations and probability theory.  In 2008 
G.Lusztig suggested an analogue of the theory of total positivity for the Lie group context  \cite{lustig1}. 

Total positivity was used  to construct solutions  of the  Kadomtsev-Petviashvili (KP) 
differential equation  in a paper by T.M.Malanyuk \cite{malan}, a joint paper of M.Boiti, F.Pempinelli, 
A.Pogrebkov \cite[section II]{pog}, and a joint paper of Y.Kodama and L.Williams \cite{kw}. 
%each matrix of appropriate dimension with positive maximal 
%minors  generates a multisoliton solution of the  Kadomtsev-Petviashvili (KP) 
%differential equation. Important results in this direction were also later obtained by Y.Kodama and L.Williams in \cite{kw}. 
Sergey Fomin's talk at the ICM-2010 
\cite{fomin} was devoted to deep relations between total positivity and cluster algebras. 
There exist several approaches of construction of totally positive matrices, see \cite{karlin, gantkr, post, bg}, 
\cite[p.290 of Russian edition]{gantkr2}. 
In the previous paper \cite{bg} we have constructed a class of explicit one-dimensional families of  totally positive matrices 
given by a finite collection of double-sided infinite vector functions, whose components are modified Bessel functions of the first kind. 
Matrices of such kind arised in a paper of V.M.Buchstaber and S.I.Tertychnyi in the construction of appropriate solutions on the non-linear differential equations in a model of overdamped Josephson junction in superconductivity, see \cite{bt2} and references therein. 
It was shown in \cite{bg} that the nature of the modified Bessel functions as coefficients of appropriate generating function 
implies that the infinite vector formed by appropriate minors of the above-mentioned matrices satisfies  
the {\it differential-difference heat equation} with positive constant potential in the Hilbert space $l_2$. 

In the present paper we provide a new construction of explicit multidimensional family of  totally positive matrices formed by a finite 
collection of one-sided infinite vector functions. This family is parametrized by a domain in $\rr^l$. Each row of the matrix 
corresponds to a coordinate $x_i$ in $\rr^l$, and its elements are modified Bessel functions of this coordinate.  

\begin{definition}  (see \cite[chapter 2, definition 1.1, p.46]{karlin}). 
A function $K(x,y)$ on a product of two totally ordered sets $X\times Y$ is called a {\it totally positive 
(strictly totally positive) kernel} of order $r\in\nn$, if for every $1\leq m\leq r$, $x_1<\dots<x_m$ in $X$ and 
$y_1<\dots<y_m$ in $Y$ the determinant of the matrix $(K(x_i,y_j))_{i,j=1}^m$ is nonnegative (respectively, positive). 
\end{definition}

Recall that the modified Bessel functions $I_j(y)$  of the first kind are Laurent series coefficients for the family of analytic functions
$$g_y(z)=e^{\frac y2(z+\frac1z)}=\sum_{j=-\infty}^{+\infty}I_j(y)z^j.$$
Equivalently, they are defined by the integral formulas  
$$I_j(y)=\frac1{\pi}\int_0^{\pi}e^{y\cos\phi}\cos(j\phi) d\phi, \ j\in\zz.$$ 

\begin{theorem} \label{pos0} For every $r\in\nn$ the function 
$$K(x,j)=I_j(x), \ \ j\in\zz_{\geq0}, \ x>0$$
is a strictly totally positive kernel of order $r$ with $X=\rr_+$ and $Y=\zz_{\geq0}$. 
\end{theorem}

\begin{remark} \label{exhan} 
It is known that the  infinite matrix $(A_{km})_{k,m\in\zz}$ with $A_{km}=I_{m-k}(x)$ is totally positive for every $x>0$, 
see \cite[theorem 1.3]{bg}. Its total nonnegativity follows from  I.Schoenberg's fundamental theorem proved 
in 1948  \cite{schon2}. 
\end{remark}
Let us reformulate Theorem \ref{pos0} in a more explicit way. To do this, set 
$$X_l=\{ x=(x_1,\dots,x_l)\in\rr_+^l \ | \ x_1<x_2<\dots<x_l\};$$
$$ K_m=\{ k=(k_1,\dots,k_m)\in\zz_{\geq0}^m \ | \ k_1<k_2<\dots<k_m\}.$$
For every $x\in X_l$ and $k\in K_m$ set 
\begin{equation} A_{k,x}=(a_{ij})_{i=1,\dots,l; \ j=1,\dots m}, \ a_{ij}=I_{k_j}(x_i).\label{akx}\end{equation}
In the special case, when $l=m$, set 
\begin{equation} f_{k}(x)=\det A_{k,x}.\label{fka}\end{equation}

\begin{theorem} \label{pos} For every $m\in\nn$,  $k\in K_m$ and $x\in X_m$  one has $f_{k}(x)>0$.
\end{theorem}

Theorem \ref{pos}, which is equivalent to Theorem \ref{pos0}, will be proved in Section 2. 

\begin{corollary} For every $x=(x_1,\dots,x_l)\in X_l$ the one-sided infinite matrix  
formed by the values $a_{ij}=I_j(x_i)$, $i=1,\dots,l$, 
$j=0,1,2,\dots$ is  totally positive. 
\end{corollary}

This corollary follows immediately from the theorem.

\def\ka{\kappa_{\alpha}}
\begin{remark} \label{karl} Various necessary 
and sufficient conditions on  a kernel $K$ to be (strictly) totally positive were stated and proved in S.Karlin's book 
\cite[chapter 2]{karlin}. In the case, when $K(x,y)$ is defined on a product of two intervals and is smooth enough, 
a sufficient condition for its strict total positivity 
says that appropriate matrix formed by 
appropriate (higher order) partial derivatives of the function $K$ is positive everywhere 
\cite[chapter 2,  theorem 2.6, p. 55]{karlin}. 
(The same condition written in the form of non-strict inequality is necessary for total positivity, see 
\cite[chapter 2,  theorem 2.2, p. 51]{karlin}.)  
In \cite[chapter 3, p. 109]{karlin}  S.Karlin presented a construction of totally positive kernel coming from 
one modified Bessel function of the first kind. Namely, set 
$$\ka(x;\lambda)=\begin{cases} e^{-(x+\lambda)}(\frac x{\lambda})^{\frac{\alpha}2}I_{\alpha}(2\sqrt{x\lambda}) \text{ for } x\geq0\\
0 \text{ for } x<0\end{cases},$$
$$K_{\alpha}(x,y)=\kappa_{\alpha}(x-y; \lambda).$$
 It was shown  in loc. cit. (just after corollary 2.1) that for every $\alpha>1$ and every $r<\alpha+2$ the function 
 $K_{\alpha}(x,y)$ is a totally positive kernel of order $r$. 
 S.Karlin also presented constructions of totally positive 
matrices coming from the Green function of a given linear differential operator presented as a product of first order linear differential 
 operators of appropriate type. His constructions 
 deal with the fundamental solution $\phi(x,t)$ of the corresponding linear differential equation; 
 the function $\phi(x,t)$ depends on the parameter $t$: the value of argument where the fundamental solution jumps. 
 He had shown that appropriate 
 class of matrices of type $\phi(x_i,t_j)$ have positive determinants \cite[p.503]{karlin}. He provided a similar construction of totally positive matrices associated to  a given classical Sturm--Liouville  differential operator \cite[pp.535--538]{karlin}. Let us emphasize that 
 each totally positive matrix given by our construction includes values  of modified Bessel functions 
 with different indices, which are solutions of Sturm--Liouville equations 
with different parameters. On the other hand, each totally positive matrix from the 
above-mentioned Karlin's construction is expressed 
via either  just one given modified Bessel function, or via fundamental solutions of just one given linear equation. 
Therefore, our result is not covered by Karlin's constructions.
\end{remark}

We prove Theorem \ref{pos} by induction in $m$. For the proof of the induction step we consider the sequence of all the determinants $f_{k}(x)$ for all $k\in K_m$ as an infinite-dimensional 
vector function in new variables $y=(x_1,w)$, $w=(w_2,\dots,w_m)$, $w_j=x_j-x_1$. We fix $w$ and consider the latter vector function 
as a function of one variable $x_1\geq0$. Analogously to the arguments from 
\cite[section 2]{bg}, we show that it satisfies an ordinary differential 
equation given by a linear bounded vector field on the Hilbert space $l_2$ with coordinates $f_k$, $k\in K_m$ such that 
the positive quadrant $\{ f_k\geq0 \ | \ k\in K_m\}$ is invariant for its flow. We show that the initial value of the vector function 
 for $x_1=0$ lies in this quadrant and is non-zero. This will imply positivity of all the functions $f_k(x_1,w)$ for all $x_1>0$, as in loc. cit. 
 
 In Section 2 we present Footnotes 1 and 2 correcting coefficients in 
 two formulas from our previous paper \cite{bg}. These formulas were used in \cite{bg} in similar arguments, 
 which remain valid without changes after these corrections.

It is known that the modified Bessel functions $I_{\nu}(x)$ of the first kind are given by the series 
$$I_{\nu}(x)=(\frac12x)^{\nu}\sum_{k=0}^{\infty}\frac{(\frac14x^2)^k}{k!\Gamma(\nu+k+1)},$$
and the latter series extends them to all the real values of the index $\nu$. 

Thus, the modified Bessel functions of the first kind yield examples of totally positive kernels of two following different kinds. For every $r\in\nn$ Karlin's example, see Remark \ref{karl}, yields a totally positive kernel $K_{\alpha}(x,y)=\ka(x-y; \lambda)$ of order $r$ constructed from just one modified Bessel function $I_{\alpha}$ with arbitrary  real index $\alpha>1$ such that $r<\alpha+2$. Our main result gives other, strictly totally positive kernel $K(y,s)=I_s(y)$ depending on 
 $y\in \rr_+$ and $s\in \zz_{\geq0}$. 
 %, which appears to be a strictly totally positive kernel. 
 
\medskip

{\bf Open Question.} {\it Is it true that the determinants $f_k(x)$ in (\ref{fka}) with $x\in X_m$ are all positive for every $m\in\nn$ and every 
$k=(k_1,\dots,k_m)$ with} {\bf (may be non-integer)} $k_j$ and {\it $k_1\in\rr_{\geq0}$, $k_1<\dots<k_m$? In other terms, is it true that the kernel $K(y,s)=I_s(y)$ is strictly 
totally positive  as a function in $(y,s)\in\rr_+\times\rr_{\geq0}$?}

\subsection{A brief survey on total positivity in Grassmanian manifolds and Lie groups}

A point $L$ of Grassmanian manifold $G_{l,m}$ of $l$-subspaces in $\rr^m$, $m>l$ is  represented by an $l\times m$-matrix, whose 
lines form a basis of the  subspace represented by the point $L$. Recall that the {\it Pl\"ucker coordinates}  of the point $L$ 
are the $l$-minors of the latter matrix. The  Pl\"ucker coordinates of the point are well-defined up to multiplication by a common factor, 
and they  are considered as homogeneous coordinates representing a point of a projective space $\mathbb{RP}^N$, $N=
\left(\begin{matrix} m\\ l\end{matrix}\right)-1$. 
The Pl\"ucker coordinates  induce the Pl\"ucker  embedding of the Grassmanian manifold to  $\mathbb{RP}^N$.

\begin{definition} 
A point $L\in G_{l,m}$ is  called {\it strictly totally positive,} if one can normalize its Pl\"ucker coordinates to make all of them positive. Or equivalently, if it can be represented by a  matrix with all  the maximal minors positive. 
\end{definition}

A.E.Postnikov's paper \cite{post} deals with the 
matrices $l\times m$, $m\geq l$ of rank $l$ satisfying the condition of nonnegativity of just maximal minors. 
One of its 
main results provides an explicit combinatorial cell decomposition 
of the corresponding subset in the Grassmanian $G_{l,m}$, called the {\it totally nonnegative Grassmanian.}  The cells are coded by combinatorial types of appropriate planar networks. 
K.Talaska \cite{talaska} obtained further development and generalization of  Postnikov's result. In particular, for a given point 
of the totally nonnegative Grassmanian the results of \cite{talaska} allow to decide what is its ambient cell and what are its affine coordinates in the cell. S.Fomin and A.Zelevinsky \cite{fomzel} studied a more general notion of total 
positivity (nonnegativity) for elements of a semisimple complex Lie group with a given double Bruhat cell decomposition. 
 They have proved that the totally positive parts of the double Bruhat cells are 
bijectively parametrized by the product of the positive quadrant $\rr_+^m$ and the positive subgroup of the maximal torus. 
For other results on totally positive (nonnegative) Grassmanians see \cite{gelf}. 

Theorem \ref{pos} of the present paper implies the following corollary.

\begin{corollary} For every $l,m\in\nn$, $l<m$, and every $k\in K_m$ the mapping 
$ H_k: X_l\to G_{l, m}$ sending $x$ to  the $l$-subspace in $\rr^m$ generated by the vectors 
$$v_k(x_i)=(I_{k_1}(x_i)\dots I_{k_m}(x_i)), \ i=1,\dots,l$$
is well-defined and injective. Its image is contained in the open subset of  strictly totally positive points.
%$l$-subspaces with positive Pl\"ucker coordinates.
\end{corollary}

\begin{proof} The well-definedness and positivity of Pl\"ucker coordinates are obvious, since the $l$-minors of the matrix $A_{k,x}$ are positive, by Theorem \ref{pos}.
Let us prove  injectivity. Fix some two distinct $x,y\in X_l$. Let us show that $H_k(x)\neq H_k(y)$. 
Fix a component $y_i$ that is different from every 
component $x_j$ of the vector $x$. Then the vectors $v_k(x_1),\dots, v_k(x_l), v_k(y_i)$ are linearly independent: 
every $(l+1)$-minor of the matrix formed by them is non-zero, by Theorem \ref{pos}. Hence, $v_k(y_i)$ is not contained in the 
$l$-subspace $H_k(x)$, which is generated by the vectors $v_k(x_1),\dots,v_k(x_l)$. Thus, $H_k(y)\neq H_k(x)$. The corollary is proved.
\end{proof}

\begin{example} Consider the infinite matrix with elements  
$$A_{ms}=I_{s-m}(x), \ \ \ m, s\in\zz.$$
It is shown in \cite[theorem 1.3]{bg} that this matrix is totally positive for every $x>0$, 
see Remark \ref{exhan}: 
%and the reference therein: 
all its minors  are positive. Therefore, the subspace 
generated by any of  its $l$  rows is $l$-dimensional, and it represents a strictly totally positive point of the infinite-dimensional Grassmanian 
manifold of $l$-subspaces in the infinite-dimensional vector space. Every submatrix in $A_{ms}$ given by its $l$ rows and 
  a finite number $r>l$ of columns represents a strictly totally positive point of the finite-dimensional Grassmanian manifold $G_{lr}$. 
% The same statement holds true if we replace $A_{ms}$ 
%by its finite submatrix formed by a finite number $r\geq l$) of its columns, as in the proof of the above corollary.
\end{example}

\section{Positivity. Proof of Theorem \ref{pos}}

In the proof of Theorem \ref{pos} we use the following classical properties of the modified Bessel functions $I_j$ of the first kind, see 
\cite[section 3.7]{watson}. 
\begin{equation} I_j=I_{-j}; \label{sym}\end{equation}
\begin{equation} I_j|_{\{ y>0\}}>0; \ I_j(0)=0 \text{ for } j\neq0; \ I_0(0)>0;\label{ioo}\end{equation}
\begin{equation} I_0'=I_1; \ I_j'=\frac12(I_{j-1}+I_{j+1}).\label{differ}\end{equation}

We prove Theorem \ref{pos} by induction in $m$. 

Induction base. For $m=1$ the statement of the theorem is obvious and follows from inequality (\ref{ioo}). 

Induction step. Let the statement of the theorem  be proved for $m=m_0$. Let us prove it for $m=m_0+1$. To do this, 
consider the sequence of all the determinants $f_{k}(x)$ for all $k\in K_m$ as an infinite-dimensional vector function $(f_k(x_1,w))_{k\in K_m}$ in the new 
variables 
$$(x_1,w), \  \ w=(w_2,\dots,w_m), \ w_j=x_j-x_1; \  \ w\in X_{m-1}.$$
The next two propositions and corollary together imply that 
%for every fixed $k\in K_m$ 
the vector function $(f_{k}(x_1,w))_{k\in K_m}$ 
with fixed $w$ and variable $x_1$   is a solution of a bounded linear ordinary differential equation in the Hilbert space $l_2$ 
of infinite sequences $(f_k)_{k\in K_m}$: a phase curve of a bounded linear vector field. 
We show that the positive quadrant $\{ f_k\geq0 \ | \ k\in K_m\}\subset l_2$ is invariant under the positive flow of the latter field, 
and the initial value $(f_k(0,w))_{k\in K_m}$ lies there. 
This implies that $f_{k}(x_1,w)\geq0$ for all $x_1\geq0$, and then we easily deduce that the latter inequality is strict for $x_1>0$. 
 
The discrete Laplacian is defined on functions on graphs. We will deal with the discrete Laplacian 
$\Delta_{discr}$ acting on functions  on the Cayley graph of the additive group $\zz^m$. 
Namely, for every $j=1,\dots,m$ let $T_j$ denote the 
corresponding shift operator: 
$$(T_jf)(k)=f(k_1,\dots,k_{j-1},k_j-1, k_{j+1},\dots,k_l).$$
Then 
\begin{equation} \dd=\frac12\sum_{j=1}^m(T_j+T_j^{-1}-2).\label{lapl1}\end{equation}
Thus, one has 
$$(\Delta_{discr}f)(p)=\frac12\sum_{s=1}^m(f(p_1,\dots,p_{s-1},p_s-1,p_{s+1},\dots,p_m)$$
\begin{equation}+f(p_1,\dots,p_{s-1},p_s+1,p_{s+1},\dots,p_m))-mf(p).\label{lapl2}\end{equation}

\begin{remark} \label{restrict} 
We will deal with the class of sequences $f(k)$ with the following properties: 

(i) $f(k)=0$, whenever  $k_i=k_j$ for some $i\neq j$;

(ii) $f(k)$ is an even function in each component $k_i$. 

This class includes   $f(k)=f_{k}(x_1,w)$: statement (i) is obvious; statement (ii) follows from equality (\ref{sym}). 
In this case the discrete Laplacian is well-defined by the above formulas (\ref{lapl1}), (\ref{lapl2}) 
on the restrictions of the latter sequences  $f(k)$ to  $k\in K_m$, as in \cite[remark 2.1]{bg}. (Each 
sequence $(f(k))_{k\in K_m}$ can be extended to a sequence $(f(k))_{k\in\zz^m}$ satisfying (ii) and antisymmetric with respect to permutation of components, hence satisfying (i).) 
In more detail, it suffices to check well-definedness of formula (\ref{lapl2}) for $p=(p_1,\dots,p_m)\in K_m$ 
with $p_1=0$. All the terms of the sum in (\ref{lapl2}) except for the first one are well-defined, since they 
are numerated by indices  $k\in K_m$ and maybe some indices $k$ with equal components $k_j=k_{j+1}$, 
for which $f(k)=0$, by (i). The first term equals 
$f(-1,p_2,\dots,p_m)+f(1,p_2,\dots,p_m)=2f(1,p_2,\dots,p_m)$, by (ii), and hence, is well-defined. 
%those $k$-components, for which 
%in formula (\ref{lapl2}) written for $p\in K_m$ with $p_1=0$ the first term in the right-hand side equals 
%$f_{(p_1-1,p_2,\dots,p_m)}=f_{(-1,p_2,\dots,p_m)}=f_{(1,p_2,\dots,p_m)}$, by (ii). 
\end{remark}

\begin{proposition} (analogous to \cite[proposition 2.2]{bg}). For every $m\geq1$ and $w\in \rr^{m-1}$ the vector function $(f(x_1,k)=f_{k}(x_1,w))_{k\in K_m}$ satisfies the following linear differential-difference 
equation\footnote{In similar formula (2.7) in \cite{bg} there is a misprint: the right-hand side 
should be multiplied by $\frac 12$}:
\begin{equation} \frac{\partial f}{\partial x_1}=\Delta_{discr}f+mf.
\label{diffs}\end{equation}
%\frac12\sum_{j=1}^l(f_{(k_1,\dots,k_{j-1}, k_j-1, k_{j+1},\dots,k_l)}(x)+
%f_{(k_1,\dots,k_{j-1}, k_j+1,k_{j+1}\dots,k_l)}(x)).
\end{proposition}
\begin{proof} Each function $f_k(x_1,w)$ is the determinant of the matrix whose columns are the vector
 functions 
$V_{k_j}(x_1,w)=(I_{k_j}(x_1), I_{k_j}(x_1+w_2),\dots.I_{k_j}(x_1+w_m))$, $j=1,\dots,m$. The derivative 
$\frac{\partial f_k(x_1,w)}{\partial x_1}$ 
thus equals the sum though $j=1,\dots,m$ of the same determinants, where the column $V_{k_j}$ 
is replaced by its derivative  $\frac{\partial V_{k_j}}{\partial x_1}$. But according to (\ref{differ}), 
the latter derivative equals $\frac12(V_{k_j-1}+V_{k_j+1})$. Therefore, 
\begin{equation} \frac{\partial f}{\partial x_1}=\frac12\sum_{j=1}^m(T_j+T_j^{-1})f=(\Delta_{discr}f+mf).
\label{diffs1}\end{equation}
%This implies (\ref{diffs}).
\end{proof}
%Equation (\ref{diffs}) follows immediately from definition, equation (\ref{differ}) and Remark \ref{restrict}, as in \cite[section 2]{bg}. 

\begin{remark} \label{new} (analogous to \cite[remark 23]{bg}). For every $k\in K_m$ the $k$-th component of the right-hand side in (\ref{diffs}) is a linear combination with 
strictly positive coefficients of 
the components $f(x_1,k')$ with $k'\in K_m$ obtained from $k=(k_1,\dots,k_m)$ by adding $\pm1$ to some $k_i$. This follows from (\ref{diffs1}). 
\end{remark}

\begin{proposition} \label{term}  \cite[proposition 2.4]{bg}.   For every constant $R>1$ and every $j\geq R^2$ one has 
\begin{equation}|I_j(z)|<\frac{R^{j}}{j!} \text{ for every }  0\leq z\leq R.\label{ineqj}\end{equation}
\end{proposition}

\begin{corollary} \label{cl2} (analogous to \cite[corollary 2.6]{bg}). For every $w\in\rr^{m-1}$ and 
$x_1\in\rr$  one has 
$(f_k(x_1,w))_{k\in K_m}\in l_2$. Moreover, there exists a function $C(R)>0$ in $R>1$ such that 
\begin{equation}\sum_{k\in K_m}|f_k(x_1,w)|^2<C(R) \text{ whenever } |x_1|+|w|\leq R,\label{hilb}\end{equation}
here $|w|=|w_2|+\dots+|w_m|$. 
\end{corollary}

\begin{proof} The proof of Corollary \ref{cl2}  repeats the proof of \cite[corollary 2.6]{bg} with minor changes. Fix an $R>1$. Set 
$$M=\max_{j\in\zz, \ 0\leq z\leq R}|I_j(z)|.$$
The number $M$ is finite, by (\ref{ineqj}) and \cite[remark 2.5]{bg}. 
Recall that $0\leq k_1<\dots<k_m$ for every $k=(k_1,\dots,k_m)\in K_m$. 
 For every $k\in K_m$ one has 
\begin{equation} |f_{k}(x_1,w)|<
m!\frac{R^{k_m}}{k_m!}M^{m-1} \text{ whenever }  |x_1|+|w|\leq R.\label{ineqf}\end{equation} 
%; \text{ here } |w|=|w_1|+\dots+|w_m|
Indeed, the last column of the matrix $A_{k,x}$ consists of the values $I_{k_m}(x_i)=I_{k_m}(x_1+w_i)$, 
which are no greater than $\frac{R^{k_m}}{k_m!}$, whenever $|x_1|+|w|\leq R$, by inequality (\ref{ineqj}). 
The other matrix elements are no greater that $M$ on the same set. Therefore, the modulus $|f_{k,n}(x)|$ of its determinant  
defined as sum of $m!$ products of functions $I_j$ satisfies inequality 
(\ref{ineqf}). This implies that the sum in (\ref{hilb}) through $k\in K_m$ is no greater than\footnote{In paper \cite{bg} there is a misprint in a similar formula for the constant $C(R)$ on p.3863: in its 
right-hand side one should 
replace the factor $l!M^{l-1}$ and the term $\frac{R^{|k|_{n,\max}}}{(|k|_{n,\max})!}$ by their squares.} 
$$C(R)=(m!M^{m-1})^2\sum_{k\in K_m}\frac{R^{2k_m}}{(k_m!)^2}<+\infty.$$ 
The corollary is proved.
 \end{proof}

\begin{definition} \cite[definition 2.7]{bg}. Let $\Omega$ be the closure of an open convex subset in a Banach space. For every 
$x\in\partial\Omega$ consider the 
union of all the rays issued from $x$ that intersect $\Omega$ in at least two distinct points (including $x$). The closure of the latter union of rays 
is a convex cone, which will be here referred to, as the {\it generating cone $K(x)$.}
%\footnote{The authors believe that this definition and the next proposition 
%are well-known to specialists, but they have not found them in literature.}} 
\end{definition}

\begin{proposition} \label{inv} \cite[proposition 2.8]{bg}. Let $H$ be a Banach space, $\Omega\subset H$ be 
a convex set as above. Let $v$ be a $C^1$ vector field on a neighborhood of 
the set $\Omega$ in $H$ such that $v(x)\in K(x)$ for every $x\in\partial\Omega$. 
Then the set $\Omega$ is invariant under the flow of the field $v$: each positive semitrajectory starting at $\Omega$ is 
contained in $\Omega$. 
\end{proposition}

Now the proof of  the induction step in Theorem \ref{pos} is analogous to the argument from \cite[end of section 2]{bg}. The right-hand 
side of differential equation  (\ref{diffs}) is a bounded linear vector field on the Hilbert space $l_2$ 
of sequences $(f_k)_{k\in K_m}$. We will denote the latter vector field by $v$. Let 
$\Omega\subset l_2$ denote the ``positive quadrant'' defined by the inequalities $f_k\geq0$. For every point $y\in\partial\Omega$ the 
vector $v(y)$ lies in its generating cone $K(y)$: the components of the field $v$ are non-negative on $\Omega$, by Remark \ref{new}.  
The vector function $(f_k(x_1)=f_{k}(x_1,w))_{k\in K_m}$ in $x_1\geq0$ is an $l_2$-valued  solution of the corresponding differential equation, 
by Corollary \ref{cl2}. One has $(f_{k}(0))_{k\in K_m}\in\Omega$:
\begin{equation}f_{k}(0)=0 \text{ whenever } k_1>0; \ f_{(0,k_2,\dots,k_m)}(0)=I_0(0)f_{(k_2,\dots,k_m)}(w_2,\dots w_m)>0.\label{foo}\end{equation}
The latter equality and inequality follow from definition, the properties (\ref{ioo}) and the induction hypothesis. This together with Proposition \ref{inv} implies that 
\begin{equation} f_{k}(x_1,w)\geq0 \text{ for every } k\in K_m \text{ and } x_1\geq0.\label{nonneg}\end{equation}
 Now let us prove that the inequality is strict for all $k\in K_m$ and $x_1>0$. Indeed, let 
$f_{p}(x_0)=0$ for some $p=(p_1,\dots,p_m)\in K_m$ and $x_0>0$. All the derivatives of the function 
$f_{p}$ are non-negative, by (\ref{diffs}), Remark \ref{new} and (\ref{nonneg}). Therefore,  $f_{p}\equiv0$ 
on the segment $[0,x_0]$. This together with (\ref{diffs}), Remark \ref{new}  
and (\ref{nonneg}) implies that $f_{p'}\equiv0$ on $[0,x_0]$ for every $p'\in K_m$ obtained from $p$ by adding $\pm1$ to some component. 
We then get by induction that $f_{(0,k_2,\dots,k_m)}(0)=0$, -- a contradiction to (\ref{foo}). The proof of Theorem \ref{pos} is complete.

\section{Acknowledgments} 

The authors are grateful to A.M.Vershik, F.V.Petrov, A.I.Bufetov, P.G.Grinevich, V.G.Gorbunov, A.K.Pogrebkov, V.M.Tikhomirov and V.Yu.Protasov for helpful discussions.

%Fekete theorem (on solid minors): Gantmacher, Krein, p. 307 of Russian edition, chapter 5, section 3, theorem 8. 
%
%Fekete, M. Rendiconti di Palermo 84 (1912), 92--93. 

\end{document}